\documentclass[10pt]{article}
\usepackage{multicol}
\usepackage{amssymb}
\usepackage{amsthm}
\usepackage{amsmath}
\usepackage{graphicx}
\usepackage{tikz}
\usepackage{algorithmic}
\newtheorem{theorem}{Theorem}[section]
\newtheorem{lemma}[theorem]{Lemma}
\newtheorem{cor}[theorem]{Corollary}
\newtheorem{rmk}[theorem]{Remark}

\newtheorem{conjecture}[theorem]{Conjecture}
\newtheorem{algorithm}[theorem]{Algorithm}

\newtheorem{dfn}[theorem]{Definition}

\newtheorem{prob}{Problem}[section]

\date{}

\title{Trees with the most subtrees -- an algorithmic approach}
\author{
Xiu-Mei Zhang \\
Department of Mathematics,\\
 Shanghai Jiao Tong University, \\
800 Dongchuan road, Shanghai, 200240, P. R. China \\
\\
Xiao-Dong Zhang \\
{\tt xiaodong@sjtu.edu.cn},\\
Department of Mathematics,\\
 Shanghai Jiao Tong University, \\
800 Dongchuan road, Shanghai, 200240, P. R. China \\
\\
Daniel Gray \\
{\tt dgray1@ufl.edu} \\
Department of Mathematics,\\
University of Florida,
Gainesville, FL, 32611 \\
\\
Hua Wang \\
{\tt hwang@georgiasouthern.edu},\\
Department of Mathematical Sciences,\\
Georgia Southern University,
Statesboro, GA 30460 \\
}

\begin{document}

\maketitle{}

\begin{abstract}
When considering the number of subtrees of trees, the extremal
 structures that maximize this number among binary trees and trees
  with a given maximum degree led to some interesting facts that
   correlate to other graphical indices in applications.
   The number of subtrees in the extremal cases also form a sequence
   of numbers that are studied by number theorists.
   The structures that maximize or minimize the number
   of subtrees among general trees, binary trees and trees
   with a given maximum degree were identified before.
    Most recently, results of this nature are generalized
    to trees with a given degree sequence. In this note,
    we characterize the trees that maximize the number of
    subtrees among trees of given order and degree sequence.
     Instead of using theoretical arguments, we take an algorithmic
      approach that explicitly describes the process of achieving
       an extremal tree from any random tree. The result also
       leads to some interesting questions and provides insights
       on finding the trees close to extremal and their numbers of subtrees.
\end{abstract}



\newpage

\section{Introduction and terminology}

For a \textit{tree} $T=(V,E)$ with vertex set $V(T)$ and
edge set $E(T)$, $d_T(u)$ denotes the {\em degree} of
vertex $u$. $P_T(u,v)$ and $d_T(u,v)$ denote the {\em path}
connecting two vertices $u,v \in V(T)$ and the {\em distance}
between them. The degree sequence of a graph is the multi set
containing the degrees of all non-leaf vertices in descending order.

We denote a tree $T$ \textit{rooted} at $r \in V(T)$ as $(T,r)$.
 Let $h_T(u)=d_T(r,u)$ and $h(T) = \max _{u \in V(T)} h_T(u)$ be
 the \textit{height} of $u$ and $T$. Also let $f(T)$ denote the
  number of subtrees of $T$ and $f^{T}_{u}(X)$ denote the number
  of subtrees containing $u$ in a subgraph $X$ of $T$.
 If $h_T(u)<h_T(v)$ then we say $u$ is an \textit{ancestor}
  of $v$ or $v$ is a \textit{descendant} of $u$.
  If $h_T(u)=h_T(v)-1$ for two vertices $u$ and $v$,
   then we say $u$ is the \textit{parent} of $v$ and $v$ is the \textit{child} of $u$.
    When two vertices share the same parent we call them \textit{siblings}.
  We sometimes
  omit the subscript or superscript if there is no ambiguity.

%If $h_T(u)<h_T(v)$ then we say $u$ is an \textit{ancestor} of $v$ or $v$ is a \textit{descendant} of $u$. If $h_T(u)=h_T(v)-1$ for two vertices $u$ and $v$, then we say $u$ is the \textit{parent} of $v$ and $v$ is the \textit{child} of $u$. When two vertices share the same parent we call them \textit{siblings}.
%since these notations are needed , they may be added

Suppose that $v_1 x_1 x_2 x_3 \ldots x_n v_2$ is a path in $T$ with $v_1$ and $v_2$
 being leaves. After the removal of all edges on the path there are still connected
 components left, each containing one of $x_1, x_2, \ldots, x_n$.
 Label these components as $X_1, X_2, \ldots, X_n$, respectively. Also,
 let $X_{\leq i}$ ($X_{\geq j}$) be the component containing $x_i$ ($x_j$)
 in $T - x_ix_{i+1}$ ($T - x_{j-1}x_j$). Fig.~\ref{figure2} shows an example of
 such labellings.

\begin{figure}[ht]
\begin{center}
\begin{tikzpicture}

\draw (0,0)--(0.5,0); \draw (1.5,0)--(3.5,0); \draw (4.5,0)--(5,0);
\draw (8,0)--(9.5,0); \draw (10.5,0)--(11,0);

\draw [fill=gray!20] (2,0)--(1.75,0.5)--(2.25,0.5)--(2,0);
\draw [fill=gray!20] (3,0)--(2.75,0.5)--(3.25,0.5)--(3,0);

\draw [fill=gray!20] (8,0)--(7,0.5)--(7,-0.5)--(8,0);
\draw [fill=gray!20] (9,0)--(8.75,0.5)--(9.25,0.5)--(9,0);

\node at (2,0.75) [] {$X_i$};
\node at (3,0.75) [] {$X_{i+1}$};

\node at (7.5,0.75) [] {$X_{\leq i}$};
\node at (9,0.75) [] {$X_{i+1}$};

\node at (0,0) [circle, draw, inner sep=0.5mm, label=below:$v_1$, fill=blue!40!white] {};
\node at (1,0) [] {$\ldots$};
\node at (2,0) [circle, draw, inner sep=0.5mm, label=below:$x_i$, fill=blue!40!white] {};
\node at (3,0) [circle, draw, inner sep=0.5mm, label=below:$x_{i+1}$, fill=blue!40!white] {};
\node at (4,0) [] {$\ldots$};
\node at (5,0) [circle, draw, inner sep=0.5mm, label=below:$v_2$, fill=blue!40!white] {};

\node at (6,0) {$\Longrightarrow$};

\node at (8,0) [circle, draw, inner sep=0.5mm, label=below:$x_i$, fill=blue!40!white] {};
\node at (9,0) [circle, draw, inner sep=0.5mm, label=below:$x_{i+1}$, fill=blue!40!white] {};
\node at (10,0) [] {$\ldots$};
\node at (11,0) [circle, draw, inner sep=0.5mm, label=below:$v_2$, fill=blue!40!white] {};

\end{tikzpicture}
\end{center}
\caption{$X_{\leq i}$ in $T$}
\label{figure2}
\end{figure}
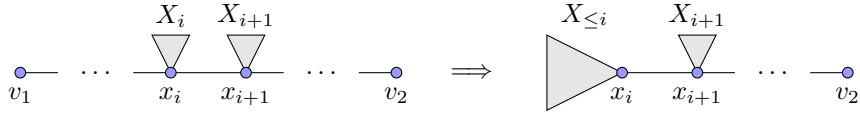

The subtrees of trees have been studied in \cite{binl} and some general
 properties were provided. A nice coincidence was found between the
 binary trees that maximize the number of subtrees and the binary trees
 that minimize the {\em Wiener index} (\cite{wiener}, defined as the sum
 of all pairwise distances between vertices), a chemical index widely
 used in biochemistry. In the same paper, formulas are given to calculate
 the number of subtrees of these extremal binary trees. The sequence of the
number of subtrees of these extremal binary trees are found to be
novel \cite{online}. These formulas use  a new  representation of
integers as a sum of powers of 2. Number theorists have already
started investigating this new binary representation \cite{dafr}.
 The results were extended to trees with a given maximum degree
\cite{kirk} and the extremal structures once again coincide with the
ones found for some
 other topological indices.  Yan and  Yeh \cite{yeh} presented an algorithm for counting
 the number of  subtrees of a tree.   The correlations of different graphical indices
 that share the same extremal structures, including the number of subtrees,
 have also been considered \cite{wagner}. Moreover, there is a relation between
The greedoid Tutte polynomial of a tree and the number of subtrees
\cite{Eisenstata}.  We will examine trees with a given {\em degree
sequence} (the descending sequence of the degrees of non-leaf
vertices).

Suppose we have a tree with given degree sequence. A \textit{greedy tree} can
be constructed in the following way:

\begin{dfn}\label{def21}
Suppose the degrees of non-leaf vertices are given, the greedy tree is achieved
 through the following `greedy algorithm':                                                                                                                   \\
\indent i) Label the vertex with the largest degree as $r$ (the root);                                  \\
\indent ii)Label the children of $r$ as $r_1, r_2,\ldots$, assign the largest
 degrees available to them such that $d(r_1) \geq d(r_2) \geq \ldots $ ;                                               \\
\indent iii) Label the children of $r_1$ as $r_{11},r_{12},\ldots $ such
that $d(r_{11}) \geq d(r_{12}) \geq \ldots $ then do the same for $r_2,r_3,\ldots $
respectively; \\
\indent iv) Repeat (iii) for all the newly labeled vertices, always start
with the children of the labeled vertex with largest degree whose neighbors
are not labeled yet.
\end{dfn}

Fig.~\ref{greedy} shows a greedy tree with degree sequence
$\{ 4, 4, 4, 3,3,3,3,3,3,3,2,2 \}$.

\begin{figure}[ht]
\begin{center}
\setlength{\unitlength}{2.5pt}
\begin{picture}(142,30)
\put(5,8){\circle*{1}}
\put(29,8){\circle*{1}}
\put(17,14){\circle*{1}}
\put(17,14){\line(2,-1){12}}
\put(17,14){\line(-2,-1){12}}
\put(41,8){\circle*{1}}
\put(65,8){\circle*{1}}
\put(53,14){\circle*{1}}
\put(65,8){\line(1,-1){5}}
\put(53,14){\line(2,-1){12}}
\put(53,14){\line(-2,-1){12}}
\put(70,3){\circle*{1}}
\put(72,3){\circle*{1}}
\put(77,8){\circle*{1}}
\put(84,3){\circle*{1}}
\put(94,3){\circle*{1}}
\put(89,8){\circle*{1}}
\put(96,3){\circle*{1}}
\put(106,3){\circle*{1}}
\put(101,8){\circle*{1}}
\put(89,14){\circle*{1}}
\put(77,8){\line(-1,-1){5}}
\put(89,8){\line(1,-1){5}}
\put(89,8){\line(-1,-1){5}}
\put(101,8){\line(1,-1){5}}
\put(101,8){\line(-1,-1){5}}
\put(89,14){\line(2,-1){12}}
\put(89,14){\line(-2,-1){12}}
\put(89,14){\line(0,-1){6}}
\put(108,3){\circle*{1}}
\put(118,3){\circle*{1}}
\put(113,8){\circle*{1}}
\put(120,3){\circle*{1}}
\put(130,3){\circle*{1}}
\put(125,8){\circle*{1}}
\put(132,3){\circle*{1}}
\put(142,3){\circle*{1}}
\put(137,8){\circle*{1}}
\put(125,14){\circle*{1}}
\put(113,8){\line(1,-1){5}}
\put(113,8){\line(-1,-1){5}}
\put(125,8){\line(1,-1){5}}
\put(125,8){\line(-1,-1){5}}
\put(137,8){\line(1,-1){5}}
\put(137,8){\line(-1,-1){5}}
\put(125,14){\line(2,-1){12}}
\put(125,14){\line(-2,-1){12}}
\put(125,14){\line(0,-1){6}}
\put(71,23){\circle*{1}}
\put(71,23){\line(2,-1){18}}
\put(71,23){\line(-2,-1){18}}
\put(71,23){\line(6,-1){54}}
\put(71,23){\line(-6,-1){54}}
\put(70,24){\makebox{$r$}}
\put(12,15){\makebox{$r_4$}}
\put(48,15){\makebox{$r_3$}}
\put(89,15){\makebox{$r_2$}}
\put(125,15){\makebox{$r_1$}}
\put(0,10){\makebox{$r_{42}$}}
\put(29,10){\makebox{$r_{41}$}}
\put(36,10){\makebox{$r_{32}$}}
\put(65,10){\makebox{$r_{31}$}}
\put(72,10){\makebox{$r_{23}$}}
\put(84,10){\makebox{$r_{22}$}}
\put(101,10){\makebox{$r_{21}$}}
\put(108,10){\makebox{$r_{13}$}}
\put(120,10){\makebox{$r_{12}$}}
\put(137,10){\makebox{$r_{11}$}}
\end{picture}
\caption{A greedy tree}
\label{greedy}
\end{center}
\end{figure}
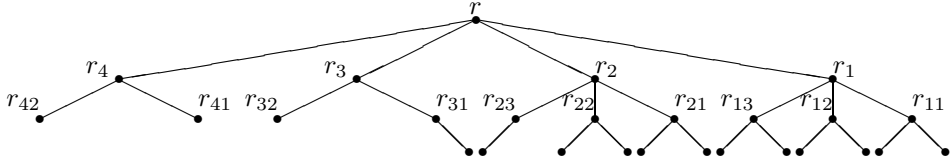

The greedy tree is shown to maximize the number of subtrees \cite{deg}.
It is nice to note that the greedy tree once again coincides with the extremal
 structures of other graphical indices. In this note, we will employ an
 algorithmic approach that defines the exact operations needed to achieve
 the extremal tree from any tree and vice versa. With this approach,
 it is possible to identify trees that are `close' to extremal and study their
 numbers of subtrees. In Section 2 we introduce several `switching operations'
 and algorithms that increase the number of subtrees in every step. In Section 3
 we prove the main result and discuss the second or $k$th extremal tree in general.
 A summary and some questions/conjectures are posted in Section 4.

\section{Algorithms on switching components}

For the following definitions, we consider a path $v_1x_1x_2 \ldots x_nv_2$
labeled as in Fig.~\ref{figure2}.

\begin{dfn}\label{def60}
A `component-switch', denoted by $S_{v_1,v_2}^T(X_i,X_j)$ (Fig.~\ref{fig-comp-switch}),
is to interchange $X_i$ and $X_j$.
\end{dfn}

\begin{figure}[ht]
\begin{center}
\begin{tikzpicture}

\draw (1,0)--(4,0);

\draw [fill=gray!20] (1,0)--(0,0.5)--(0,-0.5)--(1,0);
\draw [fill=blue!60] (2,0)--(1.75,0.5)--(2.25,0.5)--(2,0);
\draw [fill=blue!20] (3,0)--(2.75,0.5)--(3.25,0.5)--(3,0);
\draw [fill=gray!60] (4,0)--(5,0.5)--(5,-0.5)--(4,0);

\node at (-1,0) [] {$X_{\leq k-1}$};
\node at (2,0.75) [] {$X_k$};
\node at (3,0.75) [] {$X_{k+1}$};
\node at (6,0) [] {$X_{\geq k+2}$};

\node at (1,0) [circle, draw, inner sep=0.5mm, label=below:$x_{k-1}$, fill=blue!40!white] {};
\node at (2,0) [circle, draw, inner sep=0.5mm, label=below:$x_k$, fill=blue!40!white] {};
\node at (3,0) [circle, draw, inner sep=0.5mm, label=below:$x_{k+1}$, fill=blue!40!white] {};
\node at (4,0) [circle, draw, inner sep=0.5mm, label=below:$x_{k+2}$, fill=blue!40!white] {};

%After switch
\draw (1,-2)--(4,-2);

\draw [fill=gray!20] (1,-2)--(0,-1.5)--(0,-2.5)--(1,-2);
\draw [fill=blue!20] (2,-2)--(1.75,-1.5)--(2.25,-1.5)--(2,-2);
\draw [fill=blue!60] (3,-2)--(2.75,-1.5)--(3.25,-1.5)--(3,-2);
\draw [fill=gray!60] (4,-2)--(5,-1.5)--(5,-2.5)--(4,-2);

\node at (-1,-2) [] {$X_{\leq k-1}$};
\node at (2,-1.25) [] {$X_{k+1}$};
\node at (3,-1.25) [] {$X_{k}$};
\node at (6,-2) [] {$X_{\geq k+2}$};

\node at (1,-2) [circle, draw, inner sep=0.5mm, label=below:$x_{k-1}$, fill=blue!40!white] {};
\node at (2,-2) [circle, draw, inner sep=0.5mm, label=below:$x_k$, fill=blue!40!white] {};
\node at (3,-2) [circle, draw, inner sep=0.5mm, label=below:$x_{k+1}$, fill=blue!40!white] {};
\node at (4,-2) [circle, draw, inner sep=0.5mm, label=below:$x_{k+2}$, fill=blue!40!white] {};

\end{tikzpicture}
\end{center}\caption{$T$ (on top), before switching, and $S$ (at bottom), after switching}
\label{fig-comp-switch}
\end{figure}
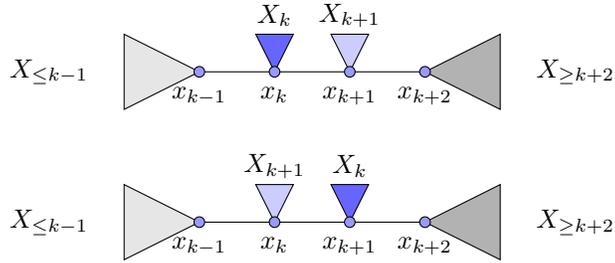

\begin{rmk}
For the purpose of better illustrating the idea, in most of the figures
we use `darker' color for `larger' component in the sense of having larger
values for $f_u^T(X)$ (for the corresponding $u$ and $X$).
\end{rmk}

\begin{dfn}
A `tail-switch', denoted by $S_{v_1,v_2}^T(X_{\leq i},X_{\geq j})$, is to
switch the components $X_{\leq i}$ and $X_{\geq j}$.
\end{dfn}

\begin{dfn}\label{dfn-degree-switch}
Suppose that $p=d_T(x_l) < d_T(x_k)=q$ in Fig.~\ref{figure-degree-switch}.
Let $x_{k,i}$ ($i=1,2, \ldots, q-2$) denote the neighbors of $x_k$ that are
 not on $P_T(v_1, v_2)$. Let $X_{k,1}, X_{k,2}, \ldots, X_{k,q-2}$ denote the
  corresponding components, ordered from smallest to largest according to the
   value of $f_{x_{k,i}}^T (X_{k,i})$. A `degree-switch'
    $S_{v_1,v_2}^T(\emptyset_{x_l},X_k''):=R_{v_1,v_2}^T(x_l,x_k)$ is to move
    the largest $q-p$ components (denoted by $X_k''$ in Fig.~\ref{figure-degree-switch})
    in $X_k$ from $x_k$ to $x_l$.
\end{dfn}

\begin{figure}[ht]
\begin{center}
\begin{tikzpicture}

\draw (1.5,0)--(2,0)--(4,0);
\draw (6,0)--(8,0)--(8.5,0);
\draw (5,-2)--(8,0); \draw (7,-2)--(8,0); \draw (8.5,-2)--(8,0); \draw (10.5,-2)--(8,0);

\draw (2,0)--(2.5,-1)--(1.5,-1)--(2,0);
\draw (8,0)--(7.7,-3)--(3.7,-2.5)--(8,0);
\draw (8,0)--(7.7,-3)--(12.5,-2.5)--(8,0);

\node at (1,0) [] {$\ldots$};
\node at (2,0) [draw, circle, label=above:$x_l$, fill=green!40] {};
\node at (5,0) [] {$\ldots$};
\node at (8,0) [draw, circle, label=above:$x_k$, fill=green!40] {};
\node at (9,0) [] {$\ldots$};

\node at (5,-2) [draw, circle, inner sep=0.5mm, label=below:$X_{k,1}$, fill=green!40] {};
\node at (6.4,-1.5) [] {$\ldots$};
\node at (7,-2) [draw, circle, inner sep=0.5mm, label=below:$X_{k,p-2}$, fill=green!40] {};
\node at (8.5,-2) [draw, circle, inner sep=0.5mm, label=below:$X_{k,p-1}$, fill=green!40] {};
\node at (9.1,-1.5) [] {$\ldots$};
\node at (10.5,-2) [draw, circle, inner sep=0.5mm, label=below:$X_{k,q-2}$, fill=green!40] {};
\node at (2,-1.5) [] {$X_l$};
\node at (6.4,-3.5) [] {$X_k'$};
\node at (9.1,-3.5) [] {$X_k''$};

\end{tikzpicture}
\end{center}
\caption{A `degree switch' when $p=d(x_l) < d(x_k)=q$}
\label{figure-degree-switch}
\end{figure}
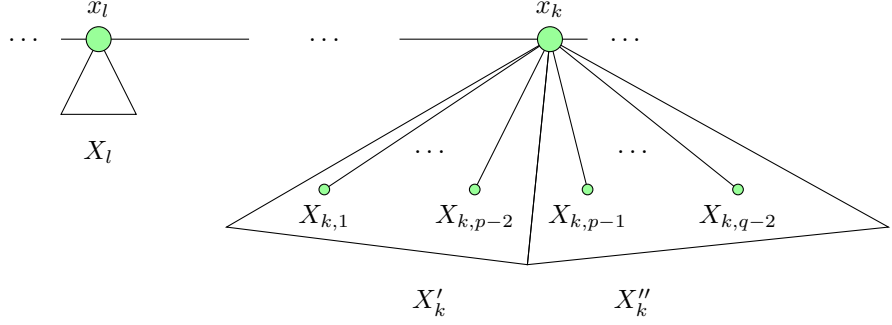

\subsection{Phase I of the switching algorithm}

For convenience let $C_k^T$, $C_{\leq k}^T$ and $C_{\geq k}^T$ denote
$f^{T}_{x_k}(X_k)$, $f^{T}_{x_{k}}(X_{\leq k})$ and   $f^{T}_{x_{k}}(X_{\geq k})$
respectively. Then $C_{\leq 0} = C_{\geq n+1} =1$.

\begin{lemma}\label{lemma-component-switch}
If $C_{\leq k-1}^T < C_{\geq k+2}^T$ and $C_k^T > C_{k+1}^T$, then performing
$S_{v_1,v_2}^T(X_k,X_{k+1})$ will increase the number of subtrees.
\end{lemma}

\begin{proof} Let $T$ and $S$ denote the trees before and after switching respectively
 as in Fig.~\ref{fig-comp-switch}.

Note that from $T$ to $S$, the numbers of subtrees that contain both or neither
$x_k$ and $x_{k+1}$ stay the same. So we only need to consider the number of
subtrees containing exactly one of $x_k$ and $x_{k+1}$.

In $T$, the number of subtrees containing $x_k$ but not $x_{k+1}$ is
$$ C_{\leq k}^T = C_{k}^T(1+C_{\leq k-1}^T) $$
and the number of subtrees containing $x_{k+1}$ but not $x_k$ is
$$ C_{\geq k+1}^T = C_{k+1}^T(1+C_{\geq k+2}^T) . $$

Similarly, these two numbers in $S$ are
$C_{k+1}^T(1+C_{\leq k-1}^T)$ and $C_{k}^T(1+C_{\geq k+2}^T)$.
Consequently, we have
$$ {f(S)-f(T)}  = (C_k^T-C_{k+1}^T)(C_{\geq k+2}^T-C_{\leq k-1}^T)  >  0 . $$
%\begin{eqnarray*}
%               {f(S)-f(T)} &= & C_{k+1}^T(1+C_{\leq k-1}^T)+C_k^T(1+C_{\geq k+2}^T)        \\
%                   &  & - C_k^T(1+C_{\leq k-1}^T)-C_{k+1}^T(1+C_{\geq k+2}^T)\\
%                                   &= & (C_k^T-C_{k+1}^T)(C_{\geq k+2}^T-C_{\leq k-1}^T)                   >  0                                                                                                                                %
%\end{eqnarray*}

\end{proof}

The following is a result of ``applying'' bubble sort algorithm with component-switches.

\begin{algorithm}\label{component-switch}
\textbf{(Phase I Switching Algorithm)}.
For a tree $T$ and leaves $v_1, v_2 \in V(T)$. Let \textbf{Label 1} be the
labeling of the path from $v_1$ to $v_2$ denoted by $v_1 x_1 x_2 \ldots x_n v_2$,
and let \textbf{Label 2} be a re-labeling of the path from $v_2$ to $v_1$ given by
 $v_2 x_1 x_2 \ldots x_n v_1$. $S = P_1^T(v_1,v_2)$ is given below.
\textnormal{\texttt{
\begin{algorithmic}
\STATE $S = (\emptyset, \emptyset)$         \hfill \COMMENT{\textnormal{\small{Empty graph}}}
\STATE $m = 0$
\STATE $T_m = T$
\WHILE{$S \neq T_m$}
    \STATE $k = 1$
    \STATE $S = T_m$        \hfill \COMMENT{\textnormal{\small{Terminate program
     if no change in $T_m$}}}
    \STATE Label $x_i$'s according to \textbf{Label 1}.
    \WHILE{$C_{\leq k-1}^{T_m} < C_{\geq k+2}^{T_m}$}
        \IF{$C_k^{T_m} > C_{k+1}^{T_m}$}
        \STATE \hfill \COMMENT{\textnormal{\small{Conditions of
         Lemma~\ref{lemma-component-switch}}}}
            \STATE $m=m+1$  \hfill \COMMENT{\textnormal{\small{Update only
            if an increase will occur}}}
            \STATE $T_m = S_{v_1,v_2}^{T_{m-1}}(X_k,X_{k+1})$
                \hfill \COMMENT{\textnormal{\small{Increases \# of subtrees}}}
            \STATE $k = k+1$
        \ELSE
            \STATE $k = k+1$
        \ENDIF
    \ENDWHILE
    \STATE $k = 1$
    \STATE Relabel $x_i$'s according to \textbf{Label 2}.
      \hfill \COMMENT{\textnormal{\small{Repeat for \textbf{Label 2}}}}
    \WHILE{$C_{\leq k-1}^{T_m} < C_{\geq k+2}^{T_m}$}
        \IF{$C_k^{T_m} > C_{k+1}^{T_m}$}
            \STATE $m=m+1$
            \STATE $T_m = S_{v_2,v_1}^{T_{m-1}}(X_k,X_{k+1})$
            \STATE $k = k+1$
        \ELSE
            \STATE $k = k+1$
        \ENDIF
    \ENDWHILE
\ENDWHILE
\end{algorithmic}
}}
\end{algorithm}

In algorithm~\ref{component-switch}, $T_m$ updates if and only if
 there is an increase in the number of subtrees. Hence the algorithm must
 terminate at some point. Simple calculations shows the following in a tree
  resulted from the Phase I algorithm.

\begin{cor}\label{cor-phi}
Let a path from $v_1$ to $v_2$ be sorted by the Phase I algorithm. Relabel
the vertices such that $x_1$ is the vertex on the path with the largest component,
 and label the rest of the vertices as in Fig.~\ref{new-path}. Without loss of generality,
  let $y_1$ and $x_2$ be labeled so that $D^T_1 \geq C^T_2$ where $C^T_i=f^{T}_{x_i}(X_i)$
   and $D^T_k=f^{T}_{y_k}(Y_k)$. Then
$$ C_1^T \geq C_2^T \geq \ldots \hbox{ and } D_1^T \geq D_2^T \geq \ldots .  $$
\end{cor}

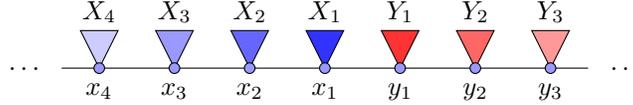
\begin{figure}[ht]
\begin{center}
\begin{tikzpicture}

\draw (-0.5,0)--(6.5,0);

\draw [fill=blue!20!white] (0,0)--(0.25,0.5)--(-0.25,0.5)--(0,0);
\draw [fill=blue!40!white] (1,0)--(1.25,0.5)--(0.75,0.5)--(1,0);
\draw [fill=blue!60!white] (2,0)--(2.25,0.5)--(1.75,0.5)--(2,0);
\draw [fill=blue!80!white] (3,0)--(3.25,0.5)--(2.75,0.5)--(3,0);
\draw [fill=red!80!white] (4,0)--(4.25,0.5)--(3.75,0.5)--(4,0);
\draw [fill=red!60!white] (5,0)--(5.25,0.5)--(4.75,0.5)--(5,0);
\draw [fill=red!40!white] (6,0)--(6.25,0.5)--(5.75,0.5)--(6,0);

\node at (0,0.75) [] {$X_4$};
\node at (1,0.75) [] {$X_3$};
\node at (2,0.75) [] {$X_2$};
\node at (3,0.75) [] {$X_1$};
\node at (4,0.75) [] {$Y_1$};
\node at (5,0.75) [] {$Y_2$};
\node at (6,0.75) [] {$Y_3$};

\node at (-1,0) [] {$\ldots$};
\node at (0,0) [circle, draw, inner sep=0.5mm, label=below:$x_{4}$, fill=blue!40!white] {};
\node at (1,0) [circle, draw, inner sep=0.5mm, label=below:$x_{3}$, fill=blue!40!white] {};
\node at (2,0) [circle, draw, inner sep=0.5mm, label=below:$x_{2}$, fill=blue!40!white] {};
\node at (3,0) [circle, draw, inner sep=0.5mm, label=below:$x_{1}$, fill=blue!40!white] {};
\node at (4,0) [circle, draw, inner sep=0.5mm, label=below:$y_{1}$, fill=blue!40!white] {};
\node at (5,0) [circle, draw, inner sep=0.5mm, label=below:$y_{2}$, fill=blue!40!white] {};
\node at (6,0) [circle, draw, inner sep=0.5mm, label=below:$y_{3}$, fill=blue!40!white] {};
\node at (7,0) [] {$\ldots$};

\end{tikzpicture}
\end{center}
\caption{Re-labeling of path after the Phase I algorithm}
\label{new-path}
\end{figure}

\subsection{Phase II of the switching algorithm}

For the lemmas that follow, we assume that our path is sorted by the Phase
 I algorithm and labeled as in Fig.~\ref{new-path}.
Firstly, the reader can easily verify the following technical observation.

\begin{lemma}\label{lemma-new}
If $C_1^T \geq D_1^T \geq C_2^T \geq \ldots \geq C_{k-1}^T \geq D_{k-1}^T$, let
$$ C:= \sum_{i=1}^{k-1} \prod_{j=1}^{i} C_{k-j}^T +
 C_{k-1}^T \ldots C_1^T \sum_{i=1}^{k-1} \prod_{j=1}^i D_j^T $$ and
$$ D:= \sum_{i=1}^{k-1} \prod_{j=1}^{i} D_{k-j}^T +
D_{k-1}^T \ldots D_1^T \sum_{i=1}^{k-1} \prod_{j=1}^i C_j^T,   $$

%$$ C:= \sum_{i=1}^{k-1} \prod_{j=1}^{i} (1+C_{k-j}^T) +
%(1+C_{k-1}^T) \ldots (1+C_1^T) \sum_{i=1}^{k-1} \prod_{j=1}^i (1+D_j^T) $$
%and
%$$ D:= \sum_{i=1}^{k-1} \prod_{j=1}^{i} (1+D_{k-j}^T) +
%(1+D_{k-1}^T) \ldots (1+D_1^T) \sum_{i=1}^{k-1} \prod_{j=1}^i (1+C_j^T) ,   $$
then $C \geq D$.

Similarly, if $C_1^T \geq D_1^T \geq C_2^T \geq \ldots \geq C_{k-1}^T
 \geq D_{k-1}^T \geq C_{k}^T$, let
$$ C' := \sum_{i=0}^{k-1} \prod_{j=0}^{i} C_{k-j}^T +
C_{k}^T \ldots C_1^T \sum_{i=1}^{k-1} \prod_{j=1}^i D_j^T $$ and
$$ D' := \sum_{i=1}^{k-1} \prod_{j=1}^{i} D_{k-j}^T +
D_{k-1}^T\ldots D_1^T \sum_{i=1}^{k} \prod_{j=1}^i C_j^T,    $$
%$$ C' := \sum_{i=0}^{k-1} \prod_{j=0}^{i} (1+C_{k-j}^T) +
%(1+C_{k}^T) \ldots (1+C_1^T) \sum_{i=1}^{k-1} \prod_{j=1}^i (1+D_j^T) $$ and
%$$ D' := \sum_{i=1}^{k-1} \prod_{j=1}^{i} (1+D_{k-j}^T) +
%(1+D_{k-1}^T) \ldots (1+D_1^T) \sum_{i=1}^{k} \prod_{j=1}^i (1+C_j^T) ,   $$
then $D' \geq C'$.
\end{lemma}

\begin{rmk}
Here $C$ and $D$ ($C'$ and $D'$) are the numbers of subtrees containing exactly
 one of $x_{k-1}$ and $y_{k-1}$ in the tree induced by the unions of $X_i$ and
  $Y_i$ ($i=1,2, \ldots , k-1$). Note that we have strict inequalities in both cases
   of Lemma~\ref{lemma-new} if
\begin{equation}\label{eq-new}
C_j > D_j \hbox{ } (D_j > C_{j+1}) \hbox{ for some $j$}.
\end{equation}
{\bf For the simplicity of statements, we  assume (\ref{eq-new}) in Lemmas~\ref{lemma63},
 \ref{lemma64}, \ref{lemma-degree-switch}}. Note that in the case of $C_j = D_j$ or
  $D_j = C_{j+1}$ for every $j$, we can simply relabel the vertices/components
  while keeping the same tree.
\end{rmk}

\begin{lemma}\label{lemma63}
If $C_1^T \geq D_1^T \geq C_2^T \geq \ldots \geq C_{k-1}^T \geq D_{k-1}^T$ and
$D_{\geq k}^T > C_{\geq k}^T$ for some $k$, then performing
 $S_{v_1,v_2}^T(X_{\geq k},Y_{\geq k})$
will increase the number of subtrees.

If $C_1^T \geq D_1^T \geq C_2^T \geq \ldots \geq D_{k-1}^T \geq C_k^T$ and
 $C_{\geq k+1}^T > D_{\geq k}^T$, then performing $S_{v_1,v_2}^T(X_{\geq K+1},Y_{\geq k})$
 will increase the number of subtrees.
\end{lemma}

\begin{proof} We show the first case (the other one is similar). Similar
 to Lemma~\ref{lemma-component-switch}, let $T$ be the tree prior to switching
 and $S$ the tree after switching. We only need to consider the subtrees that
 contain exactly one of $x_k$ and $y_k$, yielding
\begin{eqnarray*}
f(S)-f(T)   &   = & D_{\geq k}^T(1+C)   + C_{\geq k}^T(1+D) - C_{\geq k}^T(1+C) -
D_{\geq k}^T(1+D) \\
                &   =   &   (D_{\geq k}^T-C_{\geq k}^T)(C - D) + 0  >   0
\end{eqnarray*}
where $C$ and $D$ are as defined in Lemma~\ref{lemma-new}.
\end{proof}

\begin{lemma}\label{lemma64}
If $C_1^T \geq D_1^T \geq C_2^T \geq D_2^T \geq \ldots \geq D_{k-1}^T \geq C_{k}^T$,
$C_{\geq k+1}^T \geq D_{\geq k+1}^T$ and $D_{k}^T > C_{k}^T$, then performing
$S_{v_1,v_2}^T(X_k,Y_k)$ will  increase  the number of subtrees.

If $C_1^T \geq D_1^T \geq C_2^T \geq D_2^T \geq \ldots \geq C_{k}^T \geq D_{k}^T$,
 $D_{\geq k+1}^T \geq C_{\geq k+2}^T$ and $C_{k+1}^T > D_k^T$, then performing
 $S_{v_1,v_2}^T(X_{k+1},Y_k)$ will increase the number of subtrees.
\end{lemma}

\begin{proof}Similar to Lemma~\ref{lemma63}, we consider the first case. This time we have
\begin{eqnarray*}
f(S)-f(T)   &=      &   D_k^T(1+C_{\geq k+1}^T)(1 + C) + C_k^T(1+D_{\geq k+1}^T)(1 + D) \\
&           &   - C_k^T(1+C_{\geq k+1}^T)(1 + C) - D_k^T(1+D_{\geq k+1}^T)(1 + D)   \\
&=      &   (D_k^T-C_k^T)[(1+C_{\geq k+1})(1+C)-(1+D_{\geq k+1}^T)(1+D)]    >       0
\end{eqnarray*}
where $C$ and $D$ are as defined in Lemma~\ref{lemma-new}.
\end{proof}

With Lemmas~\ref{lemma63} and \ref{lemma64}, we combine component-switches and
tail-switches to sort a path. The following algorithm will terminate if the
 number of non-empty subtrees in both tails are equal to 1, indicating that the
 ends of the path have been reached.

\begin{algorithm}\label{tail-switch}
\textbf{(Phase II Switching Algorithm)}. For a tree $T$ and leaves $v_1, v_2$.
Choose $x_1$ to be the vertex with largest component on $P_T(v_1,v_2)$.
Let $y_1$ be the neighbor of $x_1$ on $P_T(v_1,v_2)$ with larger component.
Label the other vertices according to Fig.~\ref{new-path}. Then $S = P_2^T(v_1,v_2)$
is given below.
\textnormal{\texttt{
\begin{algorithmic}
\STATE $S = (\emptyset, \emptyset)$
\STATE $m = 0$
\STATE $T_m = T$
\STATE $k = 1$
\WHILE{$C_{\geq k+1}^{T_{m}} \neq 1$ OR $D_{\geq k+1}^{T_{m}} \neq 1$}
\STATE \hfill \COMMENT{\textnormal{\small{Terminate program if leaf vertex}}}
    \IF {$C_k^{T_{m}} < D_k^{T_{m}}$}
    \STATE \hfill \COMMENT{\textnormal{\small{Conditions of Lemma~\ref{lemma64}}}}
        \IF {$C_{\geq k+1}^{T_m} < D_{\geq k+1}^{T_m}$}
        \STATE \hfill \COMMENT{\textnormal{\small{Conditions of Lemma~\ref{lemma63}}}}
            \STATE $m = m+1$
            \STATE $T_m = S_{v_1,v_2}^{T_{m-1}}(X_{\geq k},Y_{\geq k})$
        \ELSE
            \STATE $m = m+1$
            \STATE $T_m = S_{v_1,v_2}^{T_{m-1}}(X_k,Y_k)$
        \ENDIF
    \ENDIF
    \IF {$D_k^{T_{m}} < C_{k+1}^{T_{m}}$}
    \STATE \hfill \COMMENT{\textnormal{\small{Repeat for $y_k$ and $x_{k+1}$}}}
        \IF {$D_{\geq k+1}^{T_m} < C_{\geq k+2}^{T_m}$}
            \STATE $m = m+1$
            \STATE $T_m = S_{v_1,v_2}^{T_{m-1}}(Y_{\geq k},X_{\geq k+1})$
        \ELSE
            \STATE $m = m+1$
            \STATE $T_m = S_{v_1,v_2}^{T_{m-1}}(Y_k,X_{k+1})$
        \ENDIF
    \ENDIF
    \STATE $k=k+1$
\ENDWHILE
\STATE $S = T_m$
\end{algorithmic}
}}
\end{algorithm}

After the Phase II algorithm, the path from $v_1$ to $v_2$ is labeled in such a way that
\begin{align*}
&   C_1^T \geq D_1^T \geq C_2^T \geq \ldots \geq C_n^T \geq D_n^T   \\
    \text{and } & C_{\geq 1}^T \geq D_{\geq 1}^T \geq C_{\geq 2}^T \geq
    \ldots \geq C_{\geq n}^T \geq D_{\geq n}^T
\end{align*}
for a path of odd length $2n-1$ (Fig.~\ref{final-path}) and
\begin{align*}
&   C_1^T \geq D_1^T \geq C_2^T \geq \ldots \geq D_n^T \geq C_{n+1}^T   \\
    \text{and } & C_{\geq 1}^T \geq D_{\geq 1}^T \geq C_{\geq 2}^T
    \geq \ldots \geq D_{\geq n}^T \geq C_{\geq n+1}^T
\end{align*}
for a path of even length $2n$.

\begin{figure}[ht]
\begin{center}
\begin{tikzpicture}

\draw (-1,0)--(0.5,0); \draw (1.5,0)--(5.5,0); \draw (6.5,0)--(8,0);

\draw [fill=blue!40] (0,0)--(-0.25,0.5)--(0.25,0.5)--(0,0);
\draw [fill=blue!70] (2,0)--(1.75,0.5)--(2.25,0.5)--(2,0);
\draw [fill=blue!100] (3,0)--(2.75,0.5)--(3.25,0.5)--(3,0);
\draw [fill=blue!85] (4,0)--(3.75,0.5)--(4.25,0.5)--(4,0);
\draw [fill=blue!55] (5,0)--(4.75,0.5)--(5.25,0.5)--(5,0);
\draw [fill=blue!25] (7,0)--(6.75,0.5)--(7.25,0.5)--(7,0);

\node at (0,0.75) [] {$X_n$};
\node at (1,0) [] {$\ldots$};
\node at (2,0.75) [] {$X_2$};
\node at (3,0.75) [] {$X_1$};
\node at (4,0.75) [] {$Y_1$};
\node at (5,0.75) [] {$Y_2$};
\node at (6,0) [] {$\ldots$};
\node at (7,0.75) [] {$Y_n$};

\node at (-1,0) [draw, circle, inner sep=0.5mm, label=below:$v_1$, fill=green!40] {};
\node at (0,0) [draw, circle, inner sep=0.5mm, label=below:$x_n$, fill=green!40] {};
\node at (2,0) [draw, circle, inner sep=0.5mm, label=below:$x_2$, fill=green!40] {};
\node at (3,0) [draw, circle, inner sep=0.5mm, label=below:$x_1$, fill=green!40] {};
\node at (4,0) [draw, circle, inner sep=0.5mm, label=below:$y_1$, fill=green!40] {};
\node at (5,0) [draw, circle, inner sep=0.5mm, label=below:$y_2$, fill=green!40] {};
\node at (7,0) [draw, circle, inner sep=0.5mm, label=below:$y_n$, fill=green!40] {};
\node at (8,0) [draw, circle, inner sep=0.5mm, label=below:$v_2$, fill=green!40] {};

\end{tikzpicture}
\end{center}
\caption{Path of length $2n-1$ after the Phase II algorithm}
\label{final-path}
\end{figure}
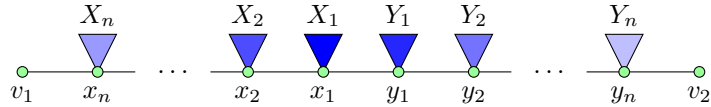

\subsection{Phase III of the switching algorithm}

The final phase of the switching algorithm uses degree-switches to sort the
 degrees of the vertices on the path.

\begin{lemma}\label{lemma-degree-switch}
Consider $P_T(v_1,v_2)$ that has been sorted by the Phase I and Phase II
algorithms and has been labeled as in Fig.~\ref{final-path}.
For $k=1,2,\ldots,n$, if $d(x_k)<d(y_k)$ or $d(y_k)<d(x_{k+1})$,
 performing a corresponding degree-switch (Definition~\ref{dfn-degree-switch}) will
  increase the number of subtrees.
\end{lemma}
\begin{proof}
We show only the case where $d(x_k) < d(y_k)$. Let $S$ be the tree after
moving $Y_k''$ to $x_k$. Also, let $D_k' = f_{y_k}^T(Y_k')$ and $D_k'' = f_{y_k}^T(Y_k'')$.

Notice that $D_k^T = D_k'D_k''$, then $f(S)-f(T)$ is
\begin{eqnarray*}
& & C^T_kD_k''(1+C^T_{\geq k+1})(1 + C) + D_k'(1+D^T_{\geq k+1})(1 + D)         \\
& & -C^T_k(1+C^T_{\geq k+1})(1 + C) - D_k'D_k''(1+D^T_{\geq k+1})(1 + D)        \\
&=& (D_k''-1)(C^T_k-D_k')[(1+C^T_{\geq k+1})(1+C)-(1+D^T_{\geq k+1})(1+D)]>0.
\end{eqnarray*}
\end{proof}

The following algorithm applies Lemma~\ref{lemma-degree-switch}. Note that
the algorithm terminates after any degree-switch since the conditions for an
 increase in the number of subtrees as a result of a degree-switch may no longer
  be certain after the switch. Consequently the path must be resorted by the
  Phase I and II algorithms after every degree-switch. However this process will
  stop since every switch increases the number of subtrees.

\begin{algorithm}
\textbf{(Phase III Switching Algorithm)} Given $P_T(v_1,v_2)$ that has been sorted
 by the Phase I and Phase II algorithms. Then $S=P_3^T(v_1,v_2)$ is given below.
\textnormal{\texttt{
\begin{algorithmic}
\FOR{$i=1$ to $n$}
    \IF{$d(x_i)<d(y_i)$}
        \STATE $S = R_{v_1,v_2}^T(x_i,y_i)$
        \STATE break \hfill \COMMENT{\textnormal{\small{Terminate algorithm}}}
    \ENDIF
    \IF{$d(y_i)<d(x_{i+1})$}
        \STATE $S = R_{v_1,v_2}^T(y_i,x_{i+1})$
        \STATE break \hfill \COMMENT{\textnormal{\small{Terminate algorithm}}}
    \ENDIF
\ENDFOR
\end{algorithmic}
}}
\end{algorithm}

\subsection{The complete algorithm}

We now introduce our final algorithm which encompasses all three previously discussed phases,
wherein every step of the algorithm increases the number of subtrees.
Therefore the algorithm will terminate after finitely many steps.

\begin{algorithm}\label{switching}
\textbf{(Switching Algorithm)} Let $T$ be a tree with leaf vertices
 $v_1, v_2, \ldots, v_l$. Then $S=SA(T)$ is given below.
\textnormal{\texttt{
\begin{algorithmic}
\STATE $S = (\emptyset, \emptyset)$
\STATE $J = T$
\STATE $m=0$
\STATE $T_m = 0$
\WHILE{$S \neq J$}
\STATE \hfill \COMMENT{\textnormal{\small{Terminate when every path is sorted}}}
    \STATE $S = J$
    \STATE $J = (\emptyset, \emptyset)$
    \FOR{$i = 1$ to $l$}
        \FOR{$j = 1$ to $l$}
            \WHILE{$J \neq T_m$}
            \STATE \hfill \COMMENT{\textnormal{\small{Terminate when path is sorted}}}
                \STATE $J=T_m$
                \STATE $T_m = P_1^{T_m}(v_i,v_j)$ \hfill \COMMENT{\textnormal{\small{Phase I}}}
                \STATE $T_m = P_2^{T_m}(v_i,v_j)$ \hfill \COMMENT{\textnormal{\small{Phase II}}}
                \STATE $m=m+1$
                \STATE $T_m = P_3^{T_{m-1}}(v_i,v_j)$ \hfill \COMMENT{\textnormal{\small{Phase III}}}
            \ENDWHILE
        \ENDFOR
    \ENDFOR
\ENDWHILE
\end{algorithmic}
}}
\end{algorithm}

\begin{rmk}\label{switching-remark}
Suppose $S = SA(T)$ where $T$ is a tree with given degree sequence.
Then for any path in $S$ from one leaf to another, we can label the vertices
on the path according to Fig.~\ref{final-path}. On such paths we have,
\begin{align*}
&   C^S_1 \geq D^S_1 \geq C^S_2 \geq \ldots \geq C^S_n \geq D^S_n, \\
&   C^S_{\geq 1} \geq D^S_{\geq 1} \geq C^S_{\geq 2} \geq \ldots
\geq C^S_{\geq n} \geq D^S_{\geq n},   \\
&    d_S(x_1) \geq d_S(y_1) \geq d_S(x_2) \geq \ldots \geq d_S(x_n) \geq d_S(y_n)
\end{align*}
for paths of odd length $2n-1$, and
\begin{align*}
&   C^S_1 \geq D^S_1 \geq C^S_2 \geq \ldots \geq D^S_n \geq C^S_{n+1}, \\
&   C^S_{\geq 1} \geq D^S_{\geq 1} \geq C^S_{\geq 2} \geq \ldots \geq D^S_{\geq n}
\geq C^S_{\geq {n+1}},   \\
&    d_S(x_1) \geq d_S(y_1) \geq d_S(x_2) \geq \ldots \geq d_S(y_n) \geq d_S(x_{n+1})
\end{align*}
for paths of even length $2n$.

If the first two conditions imply the third one, then the Phase III
algorithm would be unnecessary. We post it as a question in Section 4.
\end{rmk}

\section{The extremal trees}

The following observation is an immediate consequence from the definition of the greedy tree.
\begin{lemma}\label{gre}
A rooted tree $T$ with a given degree sequence is a greedy tree if:

i) the root $r$ has the largest degree;

ii) the heights of any two leaves differ by at most 1;

iii) for any two vertices $u$ and $w$, if $h_T(w) < h_T(u)$, then $d(w)\geq d(u)$;

iv) for any two vertices $u$ and $w$ of the same height, $d(u)> d(w) \Rightarrow d(u')
\geq d(w')$ for any
successors $u'$ of $u$ and $w'$ of $w$ that are of the same height;

v) for any two vertices $u$ and $w$ of the same height, $d(u)> d(w) \Rightarrow d(u')
\geq d(w') $ and
$ d(u'') \geq d(w'') $ for any
siblings $u'$ of $u$ and $w'$ of $w$ or successors $u''$  of $u'$ and $w''$ of $w'$ of
the same height.
\end{lemma}

\subsection{The maximal tree}

\begin{theorem} \label{theorem25}
Given the degree sequence, the greedy tree maximizes the number of subtrees.
\end{theorem}
\begin{proof}
First note that the maximal tree must satisfy the conditions in
Remark~\ref{switching-remark} or we could increase the number of subtrees.

On the other hand, we will show that these conditions in Remark~\ref{switching-remark}
 imply the properties (i -- v) of the greedy tree listed in Lemma~\ref{gre}.
 It was shown in \cite{binl} that the set of vertices that is contained in most
 subtrees consists of one or two adjacent vertices. We illustrate the idea by
 considering the first case with this vertex designated as the root.

Root the tree at the vertex $r$ contained in the most subtrees. Consider a
 path between two leaf vertices that contains $r$ and let $v$ be a vertex on the
 path adjacent to $r$. Then, since the number of subtrees containing $r$ is larger than
  the number of subtrees containing $v$, the tail  containing $r$ is larger than the
  tail containing $v$. By Remark~\ref{switching-remark}, we can label, on this path,
  the root $r$ as $x_1$ and the rest of the vertices according to Fig.~\ref{final-path}
  with $v$ labeled as $y_1$ or $x_2$ (depending on which neighbor of $r$ on this path
  has larger tail). Thus, the root has the largest degree of all vertices in the tree.

Let $u$ and $w$ be vertices such that $h(u) < h(w)$, and let $s$ be the common
ancestor of $u$ and $w$ with the greatest height. Then $h(u) = d(u,s) + h(s) < d(w,s) + h(s)$ implies that any path containing $u$ or $w$ and the root must also contain $s$. Furthermore, by Remark~\ref{switching-remark} again, $d(s) \geq d(u) \geq d(w)$. Therefore, vertices which are closer to the root will have greater degree.

This fact will force vertices which are closer to the root to be associated with the
 largest component and the largest degree on any path that we are studying. Then the
  properties (i -- v) of Lemma~\ref{gre} follows from considering the path through
  appropriate vertices.
\end{proof}

\subsection{Trees close to being maximal}

We have shown that the switching algorithm always converges to the maximal tree.
Hence, by back-tracking the algorithm, the candidates for the $k$th maximal tree
can be easily identified. Then the following corollaries are immediate consequences
 along this line:

\begin{cor}
Starting with the greedy tree, any tree with the same degree sequence can be achieved
 in a finite number of switchings wherein every switching results in a decrease in the
 number of subtrees.
\end{cor}

\begin{cor}
The $k$th maximal tree can be achieved with at most $k-1$ switches from the greedy tree.
\end{cor}

Consequently, the second maximal tree is exactly one switch away
from the greedy tree. To illustrate the insights provided by the
algorithm, we examine the second maximal tree in a bit more detail.
The following observation helps us to identify the `smaller'
switches in the sense that it produces a smaller change in the
number of subtrees.

\begin{lemma}\label{tswitch}
Given the greedy tree $T$, for any component-switch or degree-switch,
there exists a tail-switch which produces a decrease in the number of subtrees at least
 as small as that of the component-switch or degree-switch.
\end{lemma}
\begin{proof}
{\bf CASE 1}. Consider the inverse of a component-switch in Phase I and suppose that
we wish to switch $X_k$ and $X_{k+1}$ to decrease  the number of subtrees.
The switch is the same as performing a tail-switch with $X_{\leq k-1}$ and $X_{\geq k+2}$.

{\bf CASE 2}. Consider the inverse of a component-switch in Phase II as discussed in
Lemma~\ref{lemma64}, that we wish to switch $X_k$ and $Y_k$ to decrease the number
of subtrees. Note that, from the definition of a greedy tree, we already know that
each branch of $X_k$ and the tail $X_{\geq k+1}$ is at least as large as each branch
of $Y_k$ and the tail $Y_{\geq k+1}$. Consider each branch of $X_k$ as a tail and
label the tails at $x_k$ as $X_{k,1}, X_{k,2}, \ldots , X_{k,p-1} := X_{\geq k+1}$,
where $p$ is the degree of $x_k$. Similarly, label the tails at $y_k$ as
$Y_{k,1}, Y_{k,2}, \ldots , Y_{k,q-1} := Y_{\geq k+1}$ where $q$ is the degree of $y_k$.
 Let $H$ be the tree produced by the switch $S(X_k,Y_k)$ and let $H'$ be the tree
 produced by the tail switch $S(X_{k,j},Y_{k,l})$, with $X_{k,j}$ ($Y_{k,l}$) being the
  smallest (largest) tail among $X_{k,1}, X_{k,2}, \ldots , X_{k,p-1}$ ($Y_{k,1},
   Y_{k,2}, \ldots , Y_{k,q-1}$) (Fig.~\ref{figure-component-tail-smaller}). Then

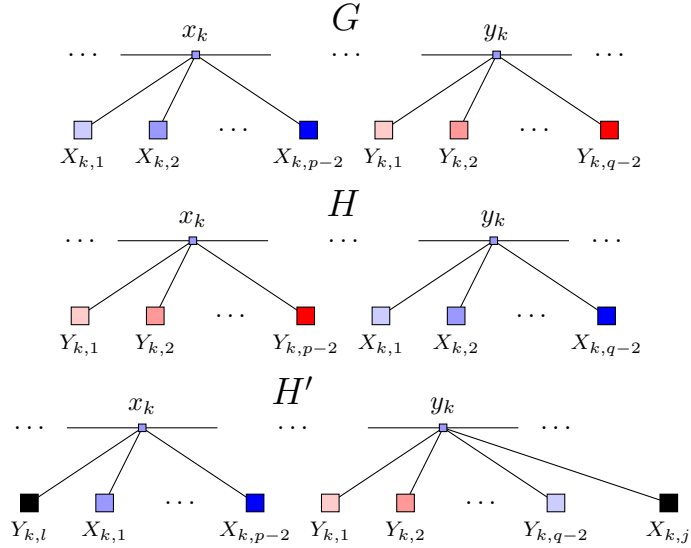
\begin{figure}[ht]
\begin{center}
\begin{tikzpicture}

%G Before switch
\draw (1,0)--(3,0);
\draw (5,0)--(7,0);

\draw (0.5,-1)--(2,0);
\draw (1.5,-1)--(2,0);
\draw (3.5,-1)--(2,0);

\draw (4.5,-1)--(6,0);
\draw (5.5,-1)--(6,0);
\draw (7.5,-1)--(6,0);

\node at (0.5,0) [] {$\ldots$};
\node at (4,0) [] {$\ldots$};
\node at (7.5,0) [] {$\ldots$};

\node at (4,0.5) {\Large{$G$}};

\node at (2,0) [rectangle, draw, inner sep=0.5mm, label=above:$x_{k}$, fill=blue!40!white] {};
\node at (6,0) [rectangle, draw, inner sep=0.5mm, label=above:$y_{k}$, fill=blue!40!white] {};

\node at (0.5,-1) [rectangle, draw, label=below:\footnotesize{$X_{k,1}$}, fill=blue!20!white] {};
\node at (1.5,-1) [rectangle, draw, label=below:\footnotesize{$X_{k,2}$}, fill=blue!40!white] {};
\node at (2.5,-1) {$\ldots$};
\node at (3.5,-1) [rectangle, draw, label=below:\footnotesize{$X_{k,p-2}$}, fill=blue!100!white] {};

\node at (4.5,-1) [rectangle, draw,  label=below:\footnotesize{$Y_{k,1}$}, fill=red!20!white] {};
\node at (5.5,-1) [rectangle, draw,  label=below:\footnotesize{$Y_{k,2}$}, fill=red!40!white] {};
\node at (6.5,-1) {$\ldots$};
\node at (7.5,-1) [rectangle, draw,  label=below:\footnotesize{$Y_{k,q-2}$}, fill=red!100!white] {};
\end{tikzpicture}

\begin{tikzpicture}
%H component switch
\draw (1,0)--(3,0);
\draw (5,0)--(7,0);

\draw (0.5,-1)--(2,0);
\draw (1.5,-1)--(2,0);
\draw (3.5,-1)--(2,0);

\draw (4.5,-1)--(6,0);
\draw (5.5,-1)--(6,0);
\draw (7.5,-1)--(6,0);

\node at (0.5,0) [] {$\ldots$};
\node at (4,0) [] {$\ldots$};
\node at (7.5,0) [] {$\ldots$};

\node at (4,0.5) {\Large{$H$}};

\node at (2,0) [rectangle, draw, inner sep=0.5mm, label=above:$x_{k}$, fill=blue!40!white] {};
\node at (6,0) [rectangle, draw, inner sep=0.5mm, label=above:$y_{k}$, fill=blue!40!white] {};

\node at (0.5,-1) [rectangle, draw, label=below:\footnotesize{$Y_{k,1}$}, fill=red!20!white] {};
\node at (1.5,-1) [rectangle, draw, label=below:\footnotesize{$Y_{k,2}$}, fill=red!40!white] {};
\node at (2.5,-1) {$\ldots$};
\node at (3.5,-1) [rectangle, draw, label=below:\footnotesize{$Y_{k,p-2}$}, fill=red!100!white] {};

\node at (4.5,-1) [rectangle, draw,  label=below:\footnotesize{$X_{k,1}$}, fill=blue!20!white] {};
\node at (5.5,-1) [rectangle, draw,  label=below:\footnotesize{$X_{k,2}$}, fill=blue!40!white] {};
\node at (6.5,-1) {$\ldots$};
\node at (7.5,-1) [rectangle, draw, label=below:\footnotesize{$X_{k,q-2}$}, fill=blue!100!white] {};
\end{tikzpicture}

\begin{tikzpicture}
%H' tail switch
\draw (1,0)--(3,0);
\draw (5,0)--(7,0);

\draw (0.5,-1)--(2,0);
\draw (1.5,-1)--(2,0);
\draw (3.5,-1)--(2,0);

\draw (4.5,-1)--(6,0);
\draw (5.5,-1)--(6,0);
\draw (7.5,-1)--(6,0);
\draw (9,-1)--(6,0);

\node at (0.5,0) [] {$\ldots$};
\node at (4,0) [] {$\ldots$};
\node at (7.5,0) [] {$\ldots$};

\node at (4,0.5) {\Large{$H'$}};

\node at (2,0) [rectangle, draw, inner sep=0.5mm, label=above:$x_{k}$, fill=blue!40!white] {};
\node at (6,0) [rectangle, draw, inner sep=0.5mm, label=above:$y_{k}$, fill=blue!40!white] {};

\node at (0.5,-1) [rectangle, draw, label=below:\footnotesize{$Y_{k,l}$}, fill=black!100!white] {};
\node at (1.5,-1) [rectangle, draw, label=below:\footnotesize{$X_{k,1}$}, fill=blue!40!white] {};
\node at (2.5,-1) {$\ldots$};
\node at (3.5,-1) [rectangle, draw, label=below:\footnotesize{$X_{k,p-2}$}, fill=blue!100!white] {};

\node at (4.5,-1) [rectangle, draw,  label=below:\footnotesize{$Y_{k,1}$}, fill=red!20!white] {};
\node at (5.5,-1) [rectangle, draw,  label=below:\footnotesize{$Y_{k,2}$}, fill=red!40!white] {};
\node at (6.5,-1) {$\ldots$};
\node at (7.5,-1) [rectangle, draw, label=below:\footnotesize{$Y_{k,q-2}$}, fill=blue!20!white] {};
\node at (9,-1) [rectangle, draw, label=below:\footnotesize{$X_{k,j}$}, fill=black!100!white] {};

\end{tikzpicture}
\end{center}
\caption{$G$--before switching, $H$--after component switch, $H'$--after tail switch.}
\label{figure-component-tail-smaller}
\end{figure}

\begin{eqnarray*}
\lefteqn{f(H')-f(H)}    \\
&=&         (\prod_{i=1 \atop i\neq j}^{p-1}(C_{k,i}+1))(D_{k,l}+1)(1+C) +
(\prod_{i=1 \atop i\neq l}^{q-1}(D_{k,i}+1))(C_{k,j}+1)(1+D) \\
& &         - (\prod_{i=1}^{p-2}(C_{k,i}+1))(D_{k,q-1}+1)(1+D) -
(\prod_{i=1}^{q-2}(D_{k,i}+1))(C_{k,p-1}+1)(1+C)       \\
%&=&            \prod_{i=2 \atop i \neq j}^{p-1}C_{k,i}C_{k,j}D_{k,q-1}C -
%\prod_{i=2 \atop i \neq j}^{p-1}C_{k,i}C_{k,1}D_{k,l}D   \\
%& &            + \prod_{i=1 \atop i \neq l}^{q-2}D_{k,i}D_{k,l}C_{k,1}D -
%\prod_{i=1 \atop i \neq l}^{q-2}D_{k,i}D_{k,q-1}C_{k,j}C \\
&=&         (\prod_{i=1 \atop i \neq j}^{p-2}(C_{k,i}+1)-\prod_{i=1 \atop i
\neq l}^{q-2}(D_{k,i}+1))((C_{k,p-1}+1)(D_{k,l}+1)(1+C)-(C_{k,j}+1)(D_{k,q-1}+1)(1+D))  \\
&>&         0
\end{eqnarray*}
from the fact that $C_k > D_k$ by considering two cases (we skip the
calculations). Here $C_{k,i}$ ($D_{k,i}$) is the number of subtrees
in the component $X_{k,i}$ ($Y_{k,i}$) containing the corresponding
root. Note that one could have $j=p-1$ or $l=q-1$ and the argument
is still true.

Therefore such a tail switch is at least as `small' as the component switch. Similarly
for switching $Y_k$ with $X_{k+1}$.

{\bf CASE 3}. Consider the inverse of a degree-switch in Phase III as discussed
in Lemma~\ref{lemma-degree-switch}, with $d(x_k) > d(y_k)$ ($d(y_k) > d(x_{k+1})$).
 Similar to CASE 2, one can show that switching the smallest branch of $X_k$ ($Y_k$)
 with the largest branch of $Y_k$ ($X_{k+1}$) will result in a decrease at least as small.
\end{proof}

\section{Summary and questions}

In summary, we characterize the trees with given degree sequence that maximize the
number of subtrees. The proof is displayed with an algorithm consisting a sequence of
`subtree-switchings' that increase the number of subtrees. By back-tracking from the
 maximal tree along these switchings provides some insights on how to find the second
  maximal tree and $k$th maximal tree in general. Some interesting questions arise
  in this study.

Firstly, one may be able to show that the `degree-switches' are not necessary,
by showing the following, namely that what was achieved in Phase III is forced to
be true after Phases I and II.

\begin{conjecture}
Suppose that $H$ is a tree such that for any path between leaf vertices we have
\[ C^H_1 \geq D^H_1 \geq C^H_2 \geq \ldots \geq C^H_n \geq D^H_n \]
for a path of odd length $2n-1$, and
\[ C^H_1 \geq D^H_1 \geq C^H_2 \geq \ldots \geq D^H_n \geq C^H_{n+1} \]
for a path of even length $2n$.
Then we must have that
\[ d(x_1) \geq d(y_1) \geq d(x_2) \geq \ldots \geq d(x_n) \geq d(y_n) \]
for the path of odd length, and
\[ d(x_1) \geq d(y_1) \geq d(x_2) \geq \ldots \geq d(y_n) \geq d(x_{n+1}) \]
for the path of even length.
\end{conjecture}

Secondly, as part of the effort to study trees that are close to maximal and their
 numbers of subtrees, one encounters several questions regarding the switch operations
  on a greedy tree.

\begin{prob} For two vertices $x_k$ and $y_k$ on the path $\ldots x_2x_1(z)y_1y_2\ldots $,
will one particular type of switch be smaller and which one?
\end{prob}

\begin{prob} If $w$ is a common ancestor with the greatest height to leaf vertices
 $v_1$, $v_2$, $v_3$, and $f_T(x_i) \geq f_T(y_i)\geq f_T(z_i) \geq f_T(x_{i+1})$
 for all $i$ in the labellings of $wx_1x_2 \ldots v_1$, $wy_1y_2 \ldots v_2$ and
 $wz_1z_2 \ldots v_3$. Will there always be a switch on $y_i$ and $z_i$ that is
  smaller than any switch on $x_i$ and $z_i$?
\end{prob}

\begin{prob}
In a path labeled as $\ldots x_2x_1(z)y_1y_2 \ldots $, if $j<k$, will a switch on
 $x_k$ and $y_k$ always be smaller than the same switch on $x_j$ and $y_j$?
\end{prob}

These seemingly elementary questions do not seem to have a straight forward answer.
However, the answers to these questions will provide  characterizations of the
close-to-maximal trees.

%\begin{conjecture}
%Suppose $k \geq 2$, and that switching $C_{\geq k}$ and $D_{\geq k}$ decreases
%the number of subtrees. Then switching $D_{\geq k}$ and $C_{\geq
%k+1}$ or $C_{\geq k+1}$ and $D_{\geq k+1}$ will produce a smaller decrease.
%\end{conjecture}

%\noindent {\bf Acknowledgment:}

\end{document}